\documentclass[a4paper, dvipdfmx, 12pt]{article}
\usepackage{mathrsfs}

\usepackage{ascmac}

\usepackage{amsmath, amsthm, amsfonts, amssymb}
\usepackage{color}

\newtheorem{thm}{Theorem}[section]
\newtheorem{thm*}{Theorem*} 
\newtheorem{cor}[thm]{Corollary}
\newtheorem{lem}[thm]{Lemma}

\theoremstyle{definition}
\newtheorem{defn}{Definition}[section]
\newtheorem{rem}{Remark}[section]

\numberwithin{equation}{section} \theoremstyle{remark}

\def\P{\mathbb{P}}
\def\e{\epsilon}

\begin{document}

\title{ 
Ergodicity of stochastic Cahn-Hilliard equations with logarithmic 
potentials driven by degenerate or nondegenerate noises}
\author{
\sc{Ludovic Gouden\`ege}\footnote{E-mail: ludovic.goudenege@math.cnrs.fr} \  \sc{and} \
 \sc{Bin Xie}\footnote{ E-mail:
 bxie@shinshu-u.ac.jp; bxieuniv@outlook.com}\\
{\small $^{\ast}$ F\'ed\'eration de Math\'ematiques de CentraleSup\'elec - CNRS FR3487}\\
{\small 3 rue Joliot Curie, 91190 Gif-sur-Yvette, France}\\
{\small $^\dagger$ Department of Mathematical Sciences, Faculty of
Science, Shinshu University }\\
{ \small{3-1-1 Asahi, Matsumoto, Nagano
390-8621, Japan } } }
\date{}
\maketitle


\begin{abstract}
We study the asymptotic  properties of the stochastic Cahn-Hilliard equation 
with the logarithmic free energy by establishing different 
dimension-free Harnack inequalities according to various kinds of noises.
 The main characteristics of this equation  are the singularities of
the logarithmic free energy at $1$ and $-1$ and the conservation of the mass of the solution in its spatial variable.  Both the space-time  colored noise and
the  space-time  white  noise are considered.  For the highly degenerate space-time  colored  noise, the asymptotic log-Harnack 
inequality is established under the so-called essentially elliptic conditions.
And the 
Harnack inequality with power is established for non-degenerate space-time  white  noise.
\end{abstract}




\section{Introduction} \label{sec-1}

The Cahn-Hilliard equation is initially introduced to describe the phase 
separation in a binary alloy comprising two species when the temperature is quenched from high temperature to low one 
\cite{Cah-61, CaHi-58, CaHi-71} and now  plays a very important role in material 
science, tumor growth, population dynamics, thin films and so on. 
The deterministic Cahn-Hilliard equation has been extensively 
studied after \cite{CaHi-58}, see \cite{Lan-71, NoSe-84} for 
the case of  the polynomial free energy and \cite{BlEl-91, DeDe95}
for the case of the logarithmic free energy (see \eqref{eq-1.1} below for such energy). 
The phase separation, spinodal decomposition and nucleation are also
studied by many researchers, see \cite{BaFi-93,BlMa-01, BlMa-08, MaWa-98, NoSe-84} for instance. We refer the reader 
to \cite{Mi-17} and references therein for more studies on the deterministic case.
  
On the other hand, in the presence of thermal fluctuations,  a noise 
term is naturally required and now the stochastic Cahn-Hilliard equation 
is commonly  accepted for modeling. There are many articles which have
been devoted to the mathematical study of 
the stochastic Cahn-Hilliard equation with a polynomial free energy \cite{AnFa-18, AnKa-16, 
Car-01, DaDe-96}.  On the other hand, in applications, the solution of the Cahn-Hilliard equation is explained as  the rescaled density of atoms 
or concentration of one of material's components which takes values
in $[-1,1]$.  The polynomial free energy can not ensure that the solution
satisfies the constraint and usually the logarithmic free energy can remedy such problem. 
However, different from the deterministic case 
\cite{DeDe95}, for the stochastic case, owing to the impact of noise,
the logarithmic free energy is not strong enough to 
prevent the solution from exiting $[-1,1]$. To study it, reflection 
measures are required, see \cite{DeGo-11, GuMa-15}. 

From now on, let us introduce the stochastic 
Cahn-Hilliard equation with the logarithmic free energy.
Let  $\lambda  >0$ and define $f$ by
\begin{align}\label{eq-1.1}
f(u) =\left\{
\begin{aligned}
&+\infty,  & u\leq -1,\\
& \log\left( \frac{1-u}{1+u} \right)+ \lambda u, & u \in (-1,1), \\
& -\infty, &  u\geq 1.
\end{aligned}
\right. 
\end{align}
Let $(W(t))_{t \geq 0}$ be a cylindrical Wiener process on a completed probability space 
$(\Omega, \mathcal{F}, \mathbb{P})$.
We formally consider the stochastic Cahn-Hilliard equation with singular nonlinearity and double reflections
\begin{align}\label{sch-1}
\left\{
\begin{aligned}
\frac{\partial u}{\partial t}(t, \theta)= 
& -\frac{1}{2} \frac{\partial^2 }{ \partial \theta^2}
\left\{ \frac{\partial^2 u }{ \partial \theta^2}(t, \theta) + f(u(t,\theta ))  +\eta_-(t, \theta)  - \eta_+(t, \theta) \right\} \\
& \qquad + B\dot{W}(t, \theta), \ \  t>0, \ \theta\in (0, 1),\\
u(t,0) = & u(t,1)=  \frac{\partial^3 u}{\partial x^3} (t,0)= \frac{\partial^3 u}{\partial x^3} (t,1)=0,\    t\geq 0,\\
\int_0^t\int_0^1 (1 & +u(t,\theta))\eta_- (dt d\theta) 
=\int_0^t\int_0^1 (1-u(t,\theta))\eta_+(dt d\theta) =0, \\
u(0,\theta) = &x(\theta),\   \theta\in(0, 1),
\end{aligned}
\right.
\end{align}
where the solution $u(t, x)  \in  [-1,1] \  a.s.,$ is usually 
explained as the concentration of one species with respect to 
the other, $\eta_-, \eta_+$ are two non-negative random measures and $B$ denotes some operator which be stated clearly in Sections 
\ref{sec-2} and \ref{sec-3} respectively.
It is well-known that the Cahn-Hilliard equation can  be regarded as a gradient system in 
$H^{-1}(0,1)$ with the logarithmic free energy, which is called
Ginzburg-Landau free energy 
\begin{align*}
\mathcal{E}(u) =\int_0^1 \left(  \frac12|\nabla u( \theta) |^2 + F(u( \theta))\right)d \theta,
\end{align*}
where $F$ denotes the primitive of $-f$ with $F(0) =0$, that is $$F(u)=(1+u)\log(1+u) +(1-u)\log(1-u) -\frac{\lambda}{2} u^2, \ u\in (-1,1).$$
Note that for $\lambda >2$, $F$ denotes a double well potential, which is important 
in application. 

The stochastic PDE with reflecting measures like \eqref{sch-1}  is one kind of 
random obstacle problems \cite{Zam-17}, which is first studied for 
stochastic reaction-diffusion equations \cite{DoPa-93}. Such 
equation has been used to model the fluctuations for 
$\nabla \phi$ interface models on a hard wall with or without 
conservation of the area \cite{FuOl-01, Zam-08} and hence it has
attracted many researchers' attention. But, different from the  stochastic reaction-diffusion equation, 
due to the lack of the maximum principle
for the double Laplacian, there are few researches on stochastic 
Cahn-Hilliard equations with reflecting terms, see \cite{DeZa-07} 
for the case without  nonlinear term $f$, \cite{Gou-09} for the case of logarithmic and negative 
power nonlinear terms with only one reflection at $0$. 
The stochastic Cahn-Hilliard equation \eqref{sch-1}, the main object of this paper,  is studied mainly in \cite{DeGo-11} and \cite{GuMa-15}
for different noises.  

Roughly speaking, the main goal is to establish various dimension-free Harnack 
inequalities for the Markov semigroup associated with \eqref{sch-1} driven by two kinds of noises
and then study ergodic behavior of the solution and others properties.  
The dimension-free Harnack inequality is initially introduced in \cite{Wan-97} by F.-Y. Wang to study the log-Sobolev inequality of a diffusion process on Riemannian manifolds
and then it becomes as a very powerful and effective tool to the 
study  of various important properties of diffusion semigroups or semigroup relative to stochastic (functional) partial differential equations, such 
as, Li-Yau type heat kernel bound,
hypercontractivity, ultracontractivity,   strong Feller 
property, estimates on the heat kernels and Varadhan 
type small time asymptotics \cite{BaWaYu-15-1, DaRoWa-09, Liu-09, RoWa-03, Wan-07, Wan-13, Zh-10}.

Although recently dimension-free Harnack inequalities and their applications have also been studied 
for stochastic  reaction-diffusion equations with reflections \cite{NiXi-19, Xie-19, Zh-10}, to our best knowledge, there is no
publications on stochastic Cahn-Hilliard equations.  
Therefore, in the paper, we intend to the  
study on  the  dimension-free Harnack 
inequalities for  the Markov semigroup generated  by the solutions  of \eqref{sch-1}  perturbed by two different noises
Then we study other important  properties of the Markov semigroup
 obtained as corollary of Harnack inequalities

According to the characteristics  of  noises, both the  asymptotic log-Harnack inequality and the Harnack inequality with power will  be considered.
More precisely, we first study the asymptotic log-Harnack inequality 
 for the Markov semigroup relative to \eqref{sch-1} driven by  
the highly degenerate colored noise under the so-called essentially elliptic conditions, see \cite{GuMa-15} and \cite{HaMa-06}. 
The asymptotic log-Harnack inequality is initially introduced in \cite{Xu-11} with an application to stochastic 2D Navier-Stokes equations. The most important property of 
the asymptotic log-Harnack inequality is that the asymptotic strong Feller property introduced in \cite{HaMa-06} can be deduced from it. Hence, it has been established for various stochastic (partial)
differential equations, see \cite{BaWaYu-19} for stochastic systems with infinite memory  and see \cite{LiLi-19} for 3D Leray-$\alpha$ model. 

However, as far as we know, 
there is no publication on  the asymptotic log-Harnack inequality
for stochastic Cahn-Hilliard equations like \eqref{sch-1}, even for
stochastic reaction-diffusion equations with reflections.
Since the degenerate noise is considered, 
as explained in  \cite{GuMa-15} and \cite{HaMa-06}, 
it seems impossible to obtain the strong Feller property. 
On the other hand, it is well-known that the log-Harnack inequality or
the Harnack inequality with power implies the strong Feller property, 
see Theorem 1.4.1 \cite{Wan-13}. 
Therefore, it seems impossible for us to establish the log-Harnack 
inequality in this case, and also the Harnack inequality with power. Instead of such strong  inequalities, we will show the Markov semigroup associated with 
\eqref{sch-1} satisfies the asymptotic log-Harnack inequality, which 
is a weaker version of dimension-free Harnack inequalities.
Although the asymptotic strong Feller property for \eqref{sch-1} has been proved in \cite{GuMa-15}, we give a new proof of  the asymptotic strong Feller property  under a weaker 
condition and cover partially  the corresponding result obtained in  \cite{GuMa-15}.

The second purpose of this paper is to establish  the Harnack inequality with power 
and then in particular,  the log-Harnack inequality,  for the Markov semigroup corresponding to \eqref{sch-1} with $B=(-\Delta)^{\frac12}$. It is known that  in this case,
the average of the solution $u(t)$ in its spatial variable  is conservative  in time \cite{DeGo-11}. But, the conservation
of the average makes it  more difficult to investigate the dimension-free Harnack inequalities via coupling by change of measures than
the well-studied cases of 
stochastic  partial differential equations drive by additive noises \cite{Liu-09, NiXi-19, Wan-07, Wan-17,  Zh-10}. To overcome it,
we make use of  the strategy initially  introduced in \cite{Wan-11}, in which the stochastic finite  differential equation driven by multiplicative noise is investigated.

Let us now introduce some notations, which will be used throughout this paper. 
We denote by $\langle \cdot, \cdot \rangle$ and $|\cdot |$ the canonical  inner product and the norm of $L^2(0,1)$ respectively. 
Let $A$ be the realization of 
$\frac{\partial^2}{\partial \theta^2}$ with homogeneous Neumann boundary condition  
in $L^2(0,1)$, that is, $Au =  \frac{\partial^2 u}{\partial \theta^2}$ for any 
$u \in D(A):= \{u\in H^2(0,1): \ u'(0)=u'(1)=0 \}$.
  It is known that $A$ is self-adjoint in $L^2(0,1)$ with a complete orthonormal system 
$\{e_n \}_{n=0}^\infty$ in $L^2(0,1)$, which satisfies
$e_0(\theta) \equiv 1, e_n(\theta) =\sqrt{2} \cos(n \pi \theta), n=1,2, \cdots$ and $Ae_n= -(n\pi)^2 e_n,  n=0, 1, \cdots.$

For any $\gamma \in \mathbb{R}$, let us define 
$(-A)^{\frac{\gamma}{2} } u= \sum_{n=1}^\infty (n \pi)^{\gamma} u_n e_n$ 
for any  $u=\sum_{n=0}^\infty u_n e_n$ with its domain
$$V_\gamma =D\left((-A)^{\frac{\gamma}{2}} \right):= 
\left\{ u=\sum_{n=0}^\infty u_ne_n: \ 
\sum_{n=0}^\infty (n\pi)^{2\gamma}u_n^2<\infty  \right\}.
$$ It will be endowed with the norm $\|u\|_{\gamma}=( |u|_{\gamma}^2 + \bar{u}^2)^\frac12$.
Hereafter,  $\bar{u}$ denotes the average of $u$, that is 
$\bar{u}=\langle u, e_0 \rangle$, and $|u|_{\gamma}$ denotes the
seminorm, that is,
 $|u|_{\gamma}=\left|(-A)^{\frac{\gamma}{2}}u \right|
= \left(\sum_{n=1}^\infty (n\pi)^{2\gamma}u_n^2 \right)^\frac12.$ 
In addition, we will set
 $(u,v)_\gamma = \langle (-A)^{\frac{\gamma}{2}}u, (-A)^{\frac{\gamma}{2}}v \rangle$, which is  the semiscalar product.
For simplicity of notation, we set ${\bf H}=V_{-1}$ throughout this paper. 
Let us also denote by ${\bf H}^c$ the affine space ${\bf H}^c=\{u\in {\bf H}:\  \bar{u}=c\}$. It is easy to check that ${\bf H}^c$ is a Polish space with the metric inherited form ${\bf H}$. 

The remainder of this paper is organized as follows.
In Section \ref{sec-2}, the asymptotic log-Harnack inequality for \eqref{sch-1} driven by 
highly degenerate noise is established by using the 
asymptotic coupling method and as its application, the asymptotic heat kernel estimate and the asymptotic irreducibility are mainly stated.
In Section \ref{sec-3}, the Harnack inequality with power and the 
log-Harnack inequality  for \eqref{sch-1} with $B=(-\Delta)^{\frac12}$ are obtained  and some important applications also
are described as example.

\section{Asymptotic log-Harnack inequality for the case of highly degenerate colored noise} 
\label{sec-2}
In this section, we intend to establish the asymptotic 
log-Harnack inequality relative to \eqref{sch-1} driven by  highly degenerate colored noise, which is studied in \cite{GuMa-15}, and then as its 
application, the asymptotic strong Feller property,  the asymptotic gradient estimate
and the asymptotic heat kernel estimate are studied. Moreover, our results can be  partially  applied
to the \eqref{sch-1}  with the double-well potential $F$.

Let us recall the definition of the asymptotic log-Harnack inequality precisely based on \cite{BaWaYu-19, Xu-11}.
Let $(E, d)$  be  a Polish  space  and let $B_b(E)$ be the family of bounded 
measurable functions on $E$. We denote by  $\|\phi\|_{\infty}$  the  uniform norm of $\phi \in 
B_b(E)$.
 For a function $\phi$ on $E$, we denote by $|\nabla \phi |(x)$ its 
local Lipschitz constant at $x$,  that is, 
\begin{align*}
|\nabla \phi |(x)= \limsup_{y\to x} \frac{|\phi(x) -\phi(y)|}{d(x, y)}.
\end{align*}
In addition, here and in the sequel, $\| \nabla \phi \|_{\infty}=\sup_{x\in E} |\nabla \phi |(x)$.

\begin{defn}
Let $(P_t)_{t \geq 0}$ be a Markov semigroup on  $(E, d)$. It is called  that 
$(P_t)_{t \geq 0}$  satisfies an asymptotic log-Harnack inequality if
there exist two non-negative functions $\Phi(\cdot, \cdot)$ on $E \times E$ and 
$\Psi(\cdot, \cdot, \cdot)$ on $[0, \infty) \times E \times E$ satisfying  $\Psi(\cdot, \cdot, \cdot) \to 0$ as $t\to \infty$ such that 
\begin{align*}
P_t  \log \phi(y) \leq \log P_t \phi(x) + \Phi(x, y) +\Psi(t, x, y)\|\nabla \log \phi \|_\infty,  \ t>0
\end{align*}
holds for any  $x, y \in E$ and any positive $\phi \in B_b(E)$ with $\|\nabla \log \phi \|_\infty< \infty$.
\end{defn} 
Thanks to  the Jensen inequality, it is natural to set $\Phi(x,x)=\Psi(t,x, x)=0$ for any $t\geq 0$
and $x\in E$.
It is known that one of the important applications of the asymptotic log-Harnack inequality is that it implies the asymptotic strong Feller property, see 
Proposition 1.6  \cite{Xu-11} or Theorem 2.1 \cite{BaWaYu-19}.  For the reader's convenience,  let us  recall the definition of the asymptotic strong Feller property  according to the original paper \cite{HaMa-06}.
For a pseudo-metric $d_p$ on $E$ and two probability measures $\mu_1, \mu_2$ on $E$, let us define the transportation cost $\|\mu_1- \mu_2\|_{d_p} $ by 
$$\|\mu_1- \mu_2\|_{d_p} =\inf_{ \mu\in \mathcal{C}(\mu_1, \mu_2)}
\int_{E\times E} d_p(x,y)\mu(dx,dy),
$$
where $\mathcal{C}(\mu_1, \mu_2)$ denotes the collection of  all
probability measures on $E\times E$ with marginals $\mu_1$ and 
$\mu_2$.
 We say that $\{d_n\}_{n=1}^\infty$ is a totally separating system of 
pseudo-metrics for $E$ if  for any $m<n$ and $x,y \in E$, 
$d_m(x, y) \leq d_n(x,y)$, and   for any $x\neq y$ $\lim_{n\to \infty} d_n(x, y)=1$. 
\begin{defn}[Definition 3.1 \cite{HaMa-06}]
The Markov semigroup $(P_t)_{t \geq 0}$ on  $(E, d)$ is said to 
be
asymptotically strong Feller at point $x\in E$ 
if there exist a totally separating system of 
pseudo-metrics  $\{d_n\}_{n=1}^\infty$ for $E$ and a positive  sequence 
$\{t_n\}_{n=1}^\infty$ such that 
$$
\inf_{B\in \mathcal{B}_x}\limsup_{n \to \infty} \sup_{y\in B}\|P_{t_n}1_B(x) - P_{t_n}1_B(y)\|_{d_n} =0,
$$
where $\mathcal{B}_x$ denotes the family of all open sets including 
$x$. In addition, if this property holds for any $x\in E$, then  
$(P_t)_{t \geq 0}$  is said to be asymptotically strong Feller.
\end{defn}
Let us now explain our main goal of this section in detail.
More precisely, we intend to establish the asymptotic log-Harnack inequality for the 
Markov semigroup associated with  one of the limits of  the sequence
$\{u^n\}_{n=1}^\infty$ of the solutions of the  stochastic partial differential equation studied in \cite{GuMa-15}
\begin{align}\label{spde-1}
\left\{
\begin{aligned}
\frac{\partial u^n}{\partial t}(t, \theta)= 
& -\frac{1}{2} \frac{\partial^2 }{ \partial \theta^2}
\left\{ \frac{\partial^2 u^n }{ \partial \theta^2}(t, \theta) -p_n(u^n(t,\theta )) +\lambda u^n(t, \theta) \right\}  \\
& \qquad 
 + B\dot{W}(t, \theta), \ \  t>0, \ \theta\in (0, 1),\\
u^n(t,0) = & u^n(t,1)=  \frac{\partial^3 u^n}{\partial \theta^3} (t,0)= \frac{\partial^3 u^n}{\partial \theta^3} (t,1)=0,\  t\geq 0,\\
u^n(0,\theta) = &x(\theta),\   \theta\in(0, 1),
\end{aligned}
\right.
\end{align}
where  $$p_n(u)=2\sum_{i=0}^n \frac{u^{2i+1}}{2i+1}, \ u \in \mathbb{R}$$ is a non-decreasing 
$(2n+1)$-degree  polynomial.  It is easy to show that $-p_n(u)+ \lambda u$ converges to $f(u)$ for $u \in (-1,1)$. 

In this part, we will assume that $B$ is a Hilbert-Schmidt operator from $L^2(0,1)$ to 
$\bf{H}$, which it is equivalent to the fact that $B(-A)^{-1} B^*$ is a trace class
on  $L^2{(0,1)}$. Indeed, 
\begin{align*}
\| B\|_{\mathcal{L}_{HS}}^2 &=\sum_{n=0}^\infty \|Be_n\|_{-1}^2
= \sum_{n=0}^\infty \left\langle (-A)^{-\frac12} Be_n, 
 (-A)^{-\frac12} Be_n  \right\rangle \\
& = \sum_{n=0}^\infty \left\langle B^*(-A)^{-1} Be_n, e_n 
\right\rangle
={\rm Tr} (B(-A)^{-1} B^*),
\end{align*}
where $\| \cdot \|_{\mathcal{L}_{HS}}^2$ denotes the norm of
the Hilbert-Schmidt operator from $L^2(0,1)$ to 
$\bf{H}$, $B^*$ denotes the adjoint operator of $B$ and 
{\rm Tr} denotes the trace of an operator  on $L^2{(0,1)}$.
In the following, we set $ {\rm Tr}_{-1} ={\rm Tr}(B(-A)^{-1} B^*)$. 
In addition, to consider the ergodic property, we assume\medskip

{\bf (A1)}: $B^*e_0=0$.
\begin{rem}
To study the ergodic property, {\bf (A1)} is necessary. In fact,
it is easy to show that $\overline{u^n}(t) =\bar{x} + 
\langle B^*e_0, W(t) \rangle.$ Thus, if {\bf (A1)} fails, then there cannot be  have a stationary solution. There is no fixed mass $c$ and there is no 
invariant measure on $\bf{H}^{c}$. 
\end{rem}

Using the  notations introduced in Section \ref{sec-1}, the SPDE
 \eqref{spde-1} can be rewritten in its abstract form as below.
\begin{align}\label{spde-2}
\left\{
\begin{aligned}
du^n(t)= 
& -\frac{1}{2} A
\big\{ A u^n (t) -p_n(u^n(t)) +\lambda u^n(t) \big\}dt  + Bd{W}(t), \ \  t>0, \\
u^n(0) = &x.
\end{aligned}
\right.
\end{align}
 It is known that for each $n\in \mathbb{N}$, \eqref{spde-1} has a unique mild
 (or weak) solution $u^n$ satisfying $u^n \in C([0,\infty); H) \cap L^{2n+2} ((0,\infty) \times (0,1) ) \ a.s.$, see  
\cite{DaDe-96} or \cite{GuMa-15}.  We also know that the average of  
$u^n(t)$ is conservative, that is, $\overline{u^n}(t) =\bar{x}\ a.s.$ because of the assumption {\rm \bf (A1)}.  

Hence,  we know that \eqref{spde-1} develops in the affine space 
${\bf H}^c$  if the average $\bar{x}$ of the initial datum $x$ equals to $c$.

For each $c \in \mathbb{R}$,
let denote by $(P_t^{n, c})_{t\geq 0}$ the Markov semigroup determined by  \eqref{spde-1}, that is,
\begin{align*}
P_t^{n, c}\phi(x)= \mathbb{E}[ \phi(u^n(t; x)) ],  \ t \geq 0, x\in {\bf H}^c, \phi \in B_b({\bf H}^c),
\end{align*}
Here and in the sequel,  to specify the initial value $x$, we use 
$u^n(t; x)$ to denote the solution of \eqref{spde-2}.
 
 The  following theorem is summarized  from Proposition 3.3 and Theorem 4.1 \cite{GuMa-15}.
\begin{thm} \label{thm-1.1}
Under all of the above assumptions, for any $c\in (-1,1)$, the following results hold.\\
{\rm (i)}   There exists a subsequence $\{n_k\}$ and a Markov semigroup $(P_t^c)_{t\geq 0}$ such that 
\begin{align*}
\lim_{k \to \infty} P_t^{c, n_k} \phi(x) =P_t^c\phi(x)
\end{align*}
holds for any $x\in {\bf H}^c$ and any $\phi \in B_b({\bf H}^c)$. 
\\
{\rm (ii)}  $(P_t^c)_{t\geq 0}$ has an  invariant probability measure $\tilde{\mu}^c$.
\end{thm} 
In the following, we will fix a converging subsequence 
$P_t^{n_k,c}$ stated in Theorem \ref{thm-1.1}.
For simplicity, we will still use the notation $P_t^{n,c}$ and
$u^{n}(t)$ instead of  $P_t^{n_k,c}$ and $u^{n_k}(t)$.  Let
us denote by $u(t;x)$ the limit process  of $u^{n_k}(t)$, which is  the Markov process associated with $(P_t^c)_{t\geq 0}$. 
Formally speaking, 
the sequence  $\{u^n\}_{n=1}^\infty$  converges to the solution of 
\eqref{sch-1}, see \cite{GuMa-15}. But any limit of 
$\{u^n\}_{n=1}^\infty$ cannot be characterized as a solution of 
SPDEs, see Section 5, \cite{GuMa-15} for more information.  Here we show that
the invariant measure $\tilde{\mu}^c$ is exponentially integrable.

\begin{thm} \label{thm-2.2-190902}
Let $c\in (-1,1)$ and suppose the assumptions in Theorem 
\ref{thm-1.1} hold.  For any  $\varsigma>0$ satisfying
 $\pi^4 >2\varsigma \|B^*\|^2$,
where $\|B^*\|$ denotes the operator norm of $B^*$, 
then the invariant measure $\tilde{\mu}^c$ satisfies the exponential integrability 
\begin{align}\label{eq-2.6-0512}
\tilde{\mu}^c\left(\exp (\varsigma |\cdot |_{-1}^2) \right)<\infty.
\end{align}

If further $\pi^4 >\lambda$, then $\tilde{\mu}^c$ is the unique 
invariant measure and  for any  Lipschitz continuous function
$\phi \in B_b({\bf H}^c)$, 
\begin{align}\label{eq-2.7-0512}
|P_t^c \phi(x) -\tilde{\mu}^c(\phi)| \leq \|\nabla \phi\|_\infty \exp^{-(\pi^4 -\lambda)t} 
\left(|x|_{-1} + \tilde{\mu}(|\cdot|_{-1}) \right), \ x\in {\bf H}^c, t\geq 0.
\end{align}

\end{thm}
\begin{proof}
According to the proof of Proposition 3.1 \cite{GuMa-15},
we have that for each $n\in \mathbb{N}$,
$|u^n(t)|_{-1},\  t\geq 0$ is a continuous semimartingale with its  local
martingale part
\begin{align}\label{eq-2.8-190528}
M^n(t) =2\int_0^t\langle B^*u^n (s), dW(s) \rangle, \ t\geq 0.
\end{align} 
Moreover, the estimate 
\begin{align*}
d|u^n(t)|_{-1}^2 \leq \left(-|u^n(t)|_{1}^2 +P_c(\lambda) \right)dt  +2dM^n(t), \ t\geq 0\ a.s.
\end{align*}
is proved in the proof of Proposition 3.1 \cite{GuMa-15}, 
where  $P_c(\lambda)$ is a positive constant depending on $c, \lambda$ and ${\rm Tr}_{-1}$, but  independent of $n$. 

Noting that 
$|x|_{1} \geq \pi^2 |x|_{-1}, x\in V_{1}$, from the above inequality, it follows that 
\begin{align}\label{eq-2.7-0511}
d|u^n(t)|_{-1}^2 \leq \left(-\pi^4 |u^n(t)|_{-1}^2 +P_c(\lambda) \right)dt  +2dM^n(t), \ t\geq 0\ a.s.
\end{align}
Let $\tau_m^n=\inf\{t\geq 0: |u^n(t)|_{-1}\geq m \}, m \in \mathbb{N}$ 
be the sequence of stopping times. 
Then it is easy to show $\lim_{m \to \infty}\tau_m^n = \infty \ a.s.$ and $M^n(t\wedge \tau_m^n), t\geq 0$ is a square integrable continuous martingale. 
 Applying the It\^{o}'s formula and
using \eqref{eq-2.8-190528} and \eqref{eq-2.7-0511}, we have
\begin{align}\label{eq-2.8-0511}
& d \exp(\varsigma |u^n(t)|_{-1}^2)\\
 \leq & 
\varsigma  \exp(\varsigma |u^n(t)|_{-1}^2) (-\pi^4  |u^n(t)|_{-1}^2+P_c(\lambda) )  dt  \notag \\
& + 2\varsigma  \exp(\varsigma |u^n(t)|_{-1}^2) dM^n(t)
+ 2\varsigma^2  \exp(\varsigma |u^n(t)|_{-1}^2) |B^* u^n(t)|^2dt
  \notag 
\\
  \leq & 
\varsigma  \exp(\varsigma |u^n(t)|_{-1}^2) 
\{(-\pi^4 +{ 2\varsigma }\|B^*\|^2  )
 |u^n(t)|_{-1}^2+P_c(\lambda) 
\}dt     \notag  \\
& + 2\varsigma  \exp(\varsigma |u^n (t)|_{-1}^2) dM^n(t), \ t\leq T \wedge \tau_m^n.   \notag 
\end{align}
Combining  the fact that 
$\pi^4 >2\varsigma \|B^*\|^2$ with \eqref{eq-2.8-0511}, 
we obtain that 

{there exists a positive constant 
$K=K(\varsigma, \|B^*\|, P_c(\lambda))$ independent of $m, n$ and $t$ such that 
}
\begin{align}\label{eq-2.9-0511}
 d \exp(\varsigma |u^n(t)|_{-1}^2)
 \leq
& \left\{ 
K
 -\varsigma (\pi^4 -2 \varsigma \|B^*\|^2 ) \exp(\varsigma |u^n(t)|_{-1}^2)  \right\}dt
  \\
& + 2\varsigma  \exp(\varsigma |u^n(t)|_{-1}^2) dM^n(t), \ t\leq T \wedge \tau_m^n.   \notag 
\end{align}
To choose the constant $K$ in the above inequality, the following fundamental  inequality  is utilized:\\
For any fixed $a,b>0$, there exists a constant 
$c=c(a,b)>0$,
such that
$$(-ax + b) e^x \leq -a e^x +c, \ x\geq 0.$$
Now, noting $\varsigma >0$ and then taking $a=\pi^4 -\varsigma \|B^*\|^2 $, $b=P_c(\lambda)$, we can choose the proper constant $K$.

Thus, integrating both sides of \eqref{eq-2.9-0511} form $0$ to
$T \wedge \tau_m^n$, taking expectations, bounding nonpositive term, we obtain  that
\begin{align*}
 \mathbb{E} \left[ 
\int_0^{T \wedge \tau_m^n} \exp(\varsigma |u^n (t)|_{-1}^2)  dt  \right]
\leq  
\frac{\exp(|x|_{-1}^2) +KT}  
{\varsigma (\pi^4 -2\varsigma \|B^*\|^2) }, 
\end{align*}

which by letting $m \to \infty$ gives that 
\begin{align*}
\mathbb{E} \left[ 
\int_0^{T } \exp(\varsigma |u^n (t)|_{-1}^2) dt \right]
\leq  
\frac{\exp(|x|_{-1}^2) + K T}
{\varsigma (\pi^4 -2 \varsigma \|B^*\|^2) }
\end{align*}
for all $n\in \mathbb{N}$.

Recalling that we have fixed the converging subsequence and then letting $n \to \infty$, we 
have 
\begin{align*}
\mathbb{E} \left[ 
 \int_0^{T }\exp(\varsigma |u (t)|_{-1}^2) dt 
\right]
\leq  
\frac{\exp(|x|_{-1}^2)+  KT }
{\varsigma (\pi^4 -2 \varsigma \|B^*\|^2) },
\end{align*}
which implies the desired result \eqref{eq-2.6-0512}.

Let us now give the proof of  \eqref{eq-2.7-0512}.
Under our assumptions, we can easily show the following $1$-Lipschitz continuity of $(u(t))_{t\geq 0}$ on its initial data:
\begin{align}\label{eq-2.11-0512}
|u(t;x)- u(t;y)|_{-1} \leq \exp(-(\pi^4 -\lambda) t)|x-y|_{-1}
\end{align}
holds for any $x, y\in {\bf{H}^c}, t\geq 0$. 
Here we omit its proof and
refer the reader to Lemma \ref{lem-1.3}  below for a similar discussion.
Since $\tilde{\mu}^c$ is invariant for $P_t^c$, for any  Lipschitz continuous function
$\phi \in B_b({\bf H}^c)$, we have
\begin{align*}
|P_t^c \phi(x) -\tilde{\mu}^c(\phi)| 
= & |P_t^c \phi(x) -\tilde{\mu}^c(P_t^c\phi)| \\
\leq & \int_{\bf{H}^c} |P_t^c \phi(x) -P_t^c\phi(y)| \tilde{\mu}^c(dy) \notag \\
\leq & \|\nabla \phi\|_\infty \int_{\bf{H}^c} |u(t; x)- u(t;y)|_{-1}\tilde{\mu}^c(dy). \notag 
\end{align*}
Consequently, we can easily complete the proof  of \eqref{eq-2.7-0512} 
by \eqref{eq-2.11-0512} and \eqref{eq-2.6-0512}.
\end{proof}

From now on, let us establish the asymptotic log-Harnack inequality for $P_t^c$
under the following highly degenerate condition:\medskip

\rm{\bf(A2)}: There exists a non-negative sequence $\{b_i\}_{i=1}^\infty$ such that   $Bu=\sum_{i=1}^\infty b_i \langle u, e_i \rangle e_{i}$ and there exists a big enough integer $N$ such that 
$b_i>0, \ i=1,2, \cdots, N$ and 
\begin{align} \label{eq-1.9-0413}
(N+1)^2 \pi^2 >\lambda.
\end{align}

From this assumption, it follows that 
${\rm  span}\{ e_1, \cdots, e_N\} \subset {\rm Rang}(B)$ and such setting is known as the so-called essentially elliptic condition, see
Section 4.5 \cite{HaMa-06}.
\\
Let $\Pi_l$ be the projector from  ${\bf H}^c$ into the $(N+1)$-dimension space 
${\rm  span}\{e_0, e_1, \cdots, e_N\}$, where $N$ is  the integer appearing in the above assumption
{\rm \bf (A2)}.  
Moreover, we know that 
$B$ restricted on  ${\rm  span}\{e_1, e_2, \cdots, e_N\}$ is invertible  and its inverse will be denoted by $B^{-1}$. Thus, 
the operator $B^{-1}A\Pi_l$ is well-defined from  ${\bf H}^c$ to  ${\rm  span}\{e_1, e_2, \cdots, e_N\}$ and is bounded.
Set $$\alpha =\frac12 
\min \left\{\pi^4, \left[(N+1)^2 \pi^2 -\lambda \right]
(N+1)^2\pi^2 \right\}.$$ Now we can formulate the main result of this section. 
\begin{thm} \label{thm-1.2}
Suppose the assumptions {\rm \bf (A1)}-{\rm \bf (A2)} are satisfied. Then, for any $c\in (-1,1)$, the Markov semigroup $(P_t^c)_{t\geq 0}$ satisfies the asymptotic log-Harnack inequality.
More precisely, we have that 
\begin{align} \label{eq-1.10-0414}
P_t^c  \log \phi(y) \leq & \log P_t^c \phi(x) +
 \frac{\lambda}{8\alpha} (1- \exp(-2\alpha t))\| B^{-1}A\Pi_l \|_{op}^2 |x-y|_{-1}^2 \\
& +\exp(-\alpha t)\|\nabla \log \phi \|_\infty|x-y|_{-1},  \ t>0 \notag
\end{align}
holds for any  $x, y \in {\bf H}^c$ and any positive $\phi \in B_b({\bf H}^c)$ with $\|\nabla \log \phi \|_\infty< \infty$, 
where  $\| B^{-1}A\Pi_l \|_{op}$ denotes the operator norm of
 $B^{-1}A\Pi_l$ from ${\bf H}^c$ to the $N$-dimensional space  ${\rm  span}\{e_1, e_2, \cdots, e_N\}$. 
\end{thm}
The proof of Theorem \ref{thm-1.2} will be stated after 
Lemma \ref{lem-1.3} below. Here let us first state some applications of Theorem \ref{thm-1.2}.
As we have stated, the asymptotic strong Feller property
can be immediately deduced from Theorem \ref{thm-1.2}.
Moreover,  thanks to Theorem 2.1 \cite{BaWaYu-19},  many
 other important properties of $P_t^c\phi$, such as its gradient
  estimate,   asymptotic heat kernel estimate and asymptotic irreducibility, can be deduced.
\begin{cor}
Under the assumptions of Theorem \ref{thm-1.2},  for any $c\in (-1,1)$ the following assertions  hold:\\ 
{\rm (i)} $(P_t^c)_{t\geq 0}$ is asymptotically strong Feller.
\\
{\rm (ii)} For any Lipschitz continuous function $\phi\in B_b({\bf{H}}^c)$,
\begin{align*}
|\nabla P_t^c \phi|\leq 
\left(\frac{\lambda}{4\alpha}\right)^\frac12 \| B^{-1}A\Pi_l \|_{op}
\sqrt{P_t^c \phi^2 -(P_t^c \phi)^2} + \|\nabla \phi\|_\infty 
\exp(-\alpha t).
\end{align*} 
{\rm (iii)} For any  non-negative $\phi \in B_b({\bf{H}}^c)$ with
 $\|\phi\|_\infty <\infty$ and all $x\in {\bf{H}}^c$,
\begin{align*}
\limsup_{t\to \infty} P_t^c \phi(x) \leq \log\left(\frac{\tilde{\mu}^c(\exp \phi)}
{\int_{{\bf{H}}^c} \exp(-\frac{\lambda}{8\alpha} \| B^{-1}A\Pi_l \|_{op}^2 |x-y|_{-1}^2)\tilde{\mu}^c(dy)} \right),
\end{align*}
where $\tilde{\mu}^c$ the  invariant probability measure of $P_t^c$.\\
{\rm (iv)} Suppose for some $x\in {\bf H}^c$ and 
a measurable set $A\subset {\bf H}^c$,   
$\liminf_{t \to \infty} P_t^c(x, A)>0$ holds. Then, for any $y \in {\bf H}^c$
and $\e>0$
\begin{align*}
\liminf_{t \to \infty} P_t^c(y, A_\e)>0, 
\end{align*} 
where $A_\e$ denotes the $\e$-neighborhood of $A$ in ${\bf H}^c$.
\end{cor}
\begin{proof}
For any $x, y \in {\bf H}^c$, let us set 
$$\Phi(x, y)= \frac{\lambda}{8\alpha} \| B^{-1}A\Pi_l \|_{op}^2 |x-y|_{-1}^2$$ and 
$$\Psi(t,x, y)=\exp(-\alpha t)|x-y|_{-1}.$$
 Then, it is clear that
\begin{align*}
\lim_{y\to x} \frac{\Phi(x, y)}{|x-y|_{-1}^2}
=\frac{\lambda}{8\alpha} \| B^{-1}A\Pi_l \|_{op}^2
\end{align*}
and
\begin{align*}
\lim_{y\to x} \frac{\Psi(t,x, y)}{|x-y|_{-1}} = \exp(-\alpha t).
\end{align*}
Hence, the conditions in  Theorem 2.1 $(1)$ \cite{BaWaYu-19} are satisfied, and consequently  {\rm(i)} and {\rm (ii)} can be shown by Theorem 2.1 $(1)$ \cite{BaWaYu-19}.

 On the other hand, {\rm (iii)} and 
{\rm (iv)} are the direct results from Theorem 2.1 $(2)$  and $(4)$ \cite{BaWaYu-19} respectively.
\end{proof}
\begin{rem}
{\rm (i)}
By analogy to  the proof of Proposition 1.6  \cite{Xu-11}, we can also show that for  any Lipschitz continuous function $\phi\in B_b({\bf{H}}^c)$,
\begin{align}\label{eq-1.11-0509}
|\nabla P_t^c \phi| (x)\leq \left(\frac{\lambda}{4\alpha}\right)^\frac12 \| B^{-1}A\Pi_l \|_{op} \|\phi\|_{\infty}
 + 2 \|\nabla \phi\|_\infty \exp(-\alpha t),
\end{align}  
which is a sufficient condition for the asymptotical strong Feller property, see Proposition 3.12 \cite{HaMa-06}.
Although the asymptotic strong Feller property has been proved
in Proposition 4.3 \cite{GuMa-15}, the estimate like \eqref{eq-1.11-0509} has not been proved. So a new proof is given 
for Proposition 4.3 \cite{GuMa-15} by our result. 
\\
{\rm (ii)} From the asymptotical strong Feller property, it follows that
any two different ergodic invariant measures 
must have disjoint topological supports, see Theorem 3.16 \cite{HaMa-06}.  
\\
{\rm (iii)} The uniqueness of invariant 
measures of $P_t^c$ is proved by showing the asymptotical strong Feller property and weakly topological irreducibility in \cite{GuMa-15}. 
From the proof of  Theorem \ref{thm-1.2},  we see that ``N'' in
the assumption {\bf (A2)} for the asymptotical strong Feller property 
can be chosen a little smaller than that in Proposition 4.3 \cite{GuMa-15} (because of the factor $\pi^{2}$) since their assumption was not completely optimal. 
In addition,  the uniqueness of invariant 
measures can be also easily shown by \eqref{eq-2.7-0512} under the assumption of Theorem \ref{thm-2.2-190902}.
\end{rem}

Theorem \ref{thm-1.2}  will be proved using  the
asymptotic coupling  by change of measures. Let first us construct
the asymptotic coupling.
Let us consider the coupling stochastic partial differential equation
\begin{align}\label{spde-3}
\left\{
\begin{aligned}
dv^n(t)= 
& -\frac{1}{2} A
\big\{A v^n (t) -p_n(v^n(t)) +\lambda \Pi_h v^n(t)  +\lambda \Pi_l u^n(t) \big\}dt \\
& + Bd{W}(t), \ \  t>0,\\
v^n(0) = &y,
\end{aligned}
\right.
\end{align}
where $\Pi_h =I-\Pi_l$. By the similar arguments to \eqref{spde-2}, one can show that 
\eqref{spde-3} has a unique mild solution  $v^n$ such that
$v^n \in C([0,\infty); H) \cap L^{2n+2} ((0,\infty) \times (0,1)) \ a.s.$ Furthermore, we know that the mass of 
$v^n(t)$ is conservative in $t\geq 0$ by considering the assumption {\bf (A1)}.

\begin{lem} \label{lem-1.3}
The  solution  $u^n(t; x)$ of \eqref{spde-2} and the solution $v^n(t;y)$
of \eqref{spde-3} are asymptotically coupling 
in the following sense:
\begin{align}\label{eq-1.11} 
|u^n(t; x) -v^n(t;y)|_{-1}  \leq \exp(-\alpha t) |x-y|_{-1}, \ x, y \in {\bf H}^c.
\end{align}

\end{lem}

\begin{proof}
By the density of $L^2(0,1)$ in ${\bf H}$, it is enough for us to show 
\eqref{eq-1.11} holds for any  $x, y \in L^2(0,1)$ whenever $\bar{x} =\bar{y}=c$.  For simplicity of notations, we write 
$u^n(t)$ for $u^n(t; x)$ and  respectively $v^n(t)$ for $v^n(t; y)$ 
in the following.

Let $X^n(t) =u^n(t)-v^n(t)$.  Then it is clear that 
$X^n(t)$ satisfies  
\begin{align}\label{eq-1.12}
\left\{
\begin{aligned}
dX^n(t)= 
& -\frac{1}{2} A
\big\{A X^n (t) -[p_n(u^n(t)) -p_n(v^n(t)) ] +\lambda \Pi_h X^n(t)   \big\}dt, \\
X^n(0) = &x-y.
\end{aligned}
\right.
\end{align}
Let us first point out that  $\overline{X^n}(t)=0$ for any $t\geq 0$ by the conservative properties of  $u^n(t)$ and $v^n(t)$, which will be used below.

Without loss of generality, we assume the integer $K>N$ and let us set
$$X^{n, K} (t) = \sum_{k=0}^K \langle u^n(t)-v^n(t), e_k \rangle e_k.$$ 
Then it is known that $X^{n, K} (t) \in D(A)\ a.s.$  Therefore, by \eqref{eq-1.12} and the spectral property of the operator $A$, 
\begin{align}\label{eq-1.13}
\frac{d}{dt}|X^{n, K}(t)|_{-1}^2 = &\langle AX^{n, K}(t),X^{n, K}(t) \rangle - 
\langle p_n(u^n(t)) -p_n(v^n(t)),  X^{n, K}(t)\rangle \\
&+ \lambda \langle \Pi_h X^{n, K}(t),  X^{n, K}(t) \rangle  \notag\\
= &- |X^{n, K}(t)|_{1}^2 - \langle p_n(u^n(t)) -p_n(v^n(t)),  X^{n, K}(t)\rangle \notag \\
&+ \lambda \langle \Pi_h X^{n, K}(t),  X^{n, K}(t) \rangle.  \notag
\end{align}
Let us note that  for any $u\in V_{1}$ with  $\bar{u}=0$,
 $$|u |_{1}^2 \geq \pi^2 |\Pi_l u|^2 + (N+1)^2 \pi^2 |\Pi_h u|^2.$$  
Recalling that $\overline{X^n}(t)=0$ and noting the increasing property of $p_n$,  then by  \eqref{eq-1.13}, we obtain that 
\begin{align} \label{eq-1.18-0501}
\frac{d}{dt}|X^{n,K}(t)|_{-1}^2 
\leq  & -\pi^2 |\Pi_lX^{n, K}(t)|^2 -\{ (N+1)^2 \pi^2 -\lambda\} 
|\Pi_h X^{n,K}(t)|^2.
\end{align}
Hence, using \eqref{eq-1.9-0413} in the assumption 
{\bf (A2)} and combining \eqref{eq-1.18-0501} with  the next relations 
$$|\Pi_l u|^2 \geq \pi^2 |\Pi_l u|_{-1}^2\ \text{and} \ 
|\Pi_h u|^2 \geq (N+1)^2 \pi^2 |\Pi_h u|_{-1}^2,\ u\in L^2(0,1),$$
we have that 
\begin{align*}
\frac{d}{dt}|X^{n,K}(t)|_{-1}^2 
\leq  & -\pi^4 |\Pi_lX^{n,K}(t)|_{-1}^2 - \{(N+1)^2 \pi^2 -\lambda \}
(N+1)^2\pi^2 |\Pi_h X^{n, K}(t)|_{-1}^2  \\
\leq & -2\alpha |X^{n,K}(t)|_{-1}^2.
\end{align*} 
Finally, letting $K \to \infty$ in the above inequality, we have
\begin{align*}
\frac{d}{dt}|X^{n}(t)|_{-1}^2 
\leq & -2\alpha |X^{n}(t)|_{-1}^2, \notag
\end{align*} 
which obviously implies the desired result \eqref{eq-1.11}.
\end{proof}

From now on, let us now formulate the proof of Theorem \ref{thm-1.2}.
\begin{proof}[Proof of Theorem \ref{thm-1.2}]
  
Let us set
\begin{align*}
\xi(t)=\xi^n(t):= \frac{\lambda}{2} B^{-1}A\Pi_l (u^n(t) -v^n(t)),  \ t\geq 0.
\end{align*}
Although $\xi^n(t)$ depends on $n$, we will omit the superscript $n$, because  uniform estimates on $n$ can be shown as below.
 
By Lemma \ref{lem-1.3}, it goes that
\begin{align}\label{eq-1.18-0414}
|\xi(t)| \leq & \frac{\lambda}{2}\| B^{-1}A\Pi_l \|_{op} |u^n(t) -v^n(t)|_{-1} \\
       \leq &  \frac{\lambda}{2} \| B^{-1}A\Pi_l \|_{op} \exp(-\alpha t)|x-y|_{-1}. \notag
\end{align}
Therefore, by the  Novikov condition, we have that
\begin{align*}
M(t) =\exp\left(\int_0^t \langle \xi(s), dW(s) \rangle 
- \frac12 \int_0^t |\xi(s)|^2ds  \right)
\end{align*}
is a real-valued martingale and then by the Girsanov theorem,
\begin{align*}
\widetilde{W}(t) = W(t) -\int_0^t \xi(s)ds,  \ t\geq 0
\end{align*} 
is a cylindrical Wiener process on $L^2(0,1)$
under the probability $\tilde{\mathbb{P}}$ defined by
$$
\frac{d\tilde{\mathbb{P}}}{d\P}\Big|_{\mathcal{F}_t} =M(t). 
$$
According to the definition of $\xi(t)$, we point out that 
$M(t), \tilde{W}(t)$ and $\tilde{\P}$ are depending on $n$.
For our goal, uniform estimates on $n$ should be established.
  
Now by using the stochastic processes 
$(\widetilde{W}(t))_{t\geq 0}$
 and $(\xi(t))_{t\geq 0}$,  the coupling equation \eqref{spde-3} can be rewritten as 
\begin{align}\label{eq-1.21-0413}
\left\{
\begin{aligned}
dv^n(t)= 
& -\frac{1}{2} A
\big\{ A v^n (t) -p_n(v^n(t)) +\lambda  v^n(t)  \big\}dt 
+ Bd\widetilde{W}(t), \ \  t>0,\\
v^n(0) = &y.
\end{aligned}
\right.
\end{align}
In particular, by the  uniqueness in law of the solution of \eqref{spde-3}, it is known  that the distribution  of $v^n(t)$ under $\tilde{\P}$ is same as that of $u^n(t;y)$ under $\P$.

We first note that  for any positive $\phi \in B_b({\bf H}^c)$ with $\|\nabla \log \phi \|_\infty< \infty$, 
\begin{align} \label{eq-1.22-0414}
P_t^{n,c}  \log \phi(y) = &\mathbb{E}^{\tilde{\P}} \left[\log \phi (v^n(t)) \right]\\
  = &\mathbb{E}^{\tilde{\P}}[\log \phi (u^n(t;x))] +
  \left\{\mathbb{E}^{\tilde{\P}}\left[\log \phi (v^n(t)) \right]
- \mathbb{E}^{\tilde{\P}} \left[\log \phi (u^n(t;x)) \right] \right\} \notag \\
:= & I_1^n(t) +I_2^n(t), \notag
\end{align}
where $\mathbb{E}^{\tilde{\P}}$ denotes the expectation with respect to 
$\tilde{\P}$.

Using the definition of  $\tilde{\P}$ and the martingale property of $(M(t))_{t\geq 0}$, we have that
\begin{align}\label{eq-1.23-0414}
 I_1^n(t) = & \mathbb{E}[M(t)\log \phi (u^n(t;x))]\\
   \leq &\mathbb{E}[M(t)\log M(t)] - \mathbb{E}[M(t)] \log \mathbb{E}[M(t)]+\mathbb{E}[M(t)]\log \mathbb{E}[\phi (u^n(t;x))]  \notag\\
=  &\mathbb{E}[M(t)\log M(t)] +\mathbb{E}[M(t)]\log \mathbb{E}[ \phi (u^n(t;x)) ] \notag \\
= &\mathbb{E}[M(t)\log M(t)] +\log P_t^{n, c} \phi (x),\notag
\end{align}
where  the Young inequality 
\begin{align}\label{eq-1.27-0504}
\mathbb{E}[XY] \leq 
\mathbb{E}[X\log X]-\mathbb{E}[X]\log E[X] + \mathbb{E}[X]\log\mathbb{E}[e^Y]
\end{align}
 for any  non-negative random variables $X, Y \geq 0\ a.s.$ with $\mathbb{E}[X]>0$ has be used for the second line; see Lemma 2.4 \cite{ArTaWa-09} for its  proof.

On the other hand, using \eqref{eq-1.18-0414}, we deduce that
\begin{align*}
\mathbb{E}[M(t)\log M(t)]= & \mathbb{E}^{\tilde{\P}}[\log M(t)]  \\
=&\mathbb{E}^{\tilde{\P}} \left[ \int_0^t \langle \xi(s), dW(s) \rangle
 - \frac12\int_0^t |\xi(s)|^2ds \right]
\notag\\
=&\mathbb{E}^{\tilde{\P}} \left[ \int_0^t \langle \xi(s), 
d\widetilde{W}(s) \rangle + \frac12\int_0^t |\xi(s)|^2ds \right]
\notag\\
= &\frac12 \mathbb{E}^{\tilde{\P}} \left[ \int_0^t |\xi(s)|^2ds \right] \notag\\
\leq &   \frac{\lambda}{4} \mathbb{E}^{\tilde{\P}} 
\left[\int_0^t  \| B^{-1}A\Pi_l \|_{op}^2 \exp(-2\alpha s)|x-y|_{-1}^2 ds\right] \notag \\
= & \frac{\lambda}{8\alpha} (1- \exp(-2\alpha t)) \| B^{-1}A\Pi_l \|_{op}^2 |x-y|_{-1}^2.
\end{align*}
Hence, plugging this estimate into \eqref{eq-1.23-0414}, we have
\begin{align}\label{eq-1.24-0414}
 I_1^n(t) \leq  
&\frac{\lambda}{8\alpha} (1- \exp(-2\alpha t)) \| B^{-1}A\Pi_l \|_{op}^2 |x-y|_{-1}^2+\log P_t^{n, c} \phi (x).
\end{align}

 Let us now give the required estimate for $I_2(t)$, which is easier. In fact, by Lemma \ref{lem-1.3}, we have that 
\begin{align} \label{eq-1.25-0414}
| I_2^n(t) | \leq &  \|\nabla \log \phi \|_\infty \mathbb{E}^{\tilde{\P}} 
[|u^n(t)-v^n(t)|_{-1}]  \\
\leq & 
\exp(-\alpha t) \|\nabla \log \phi \|_\infty   |x-y|_{-1}. \notag
\end{align}
Inserting \eqref{eq-1.24-0414} and \eqref{eq-1.25-0414} into 
\eqref{eq-1.22-0414}, we see that for any $n\in \mathbb{N}$
\begin{align*}
P_t^{n, c}  \log \phi(y) \leq & \log P_t^{n, c} \phi(x)
 +\frac{\lambda}{8\alpha} (1- \exp(-2\alpha t))\| B^{-1}A\Pi_l \|_{op}^2 |x-y|_{-1}^2 \\
& +  \exp(-\alpha t)\|\nabla \log \phi \|_\infty |x-y|_{-1},  \ t>0 \notag
\end{align*}
holds for any  $x, y \in {\bf H}^c$ and any positive $\phi \in B_b({\bf H}^c)$ with $\|\nabla \log \phi \|_\infty< \infty$.

Consequently,
noting that $\| B^{-1}A\Pi_l \|_{op}$ is independent of $n$ and then
using Theorem \ref{thm-1.1}, we can obtain the desired result \eqref{eq-1.10-0414} by letting $n \to \infty$. Therefore, the proof of Theorem \ref{thm-1.2} is completed.
\end{proof}


\section{Harnack inequality for the case of nondegenerate space-time white noise}
\label{sec-3}
In this section, we will intend to study the properties of the Markov semigroup generated by the SPDE \eqref{sch-1} 
for the special case of  $B=\frac{d}{d \theta}$ 
(or equivalently $B=(-A)^\frac12$, see Remark 3.1 below) 
with its domain 
$H^1(0,1)$, which is studied in \cite{DeGo-11}. 
Let us recall the definition of solution of \eqref{sch-1} according to
Definition 1.1  \cite{DeGo-11}. 

\begin{defn}
Let the initial datum $x$ be a continuous function defined on 
$[0,1]$ with its values in  $[-1,1]$, i.e., $x\in C([0,1]; [-1,1])$.  \\
$(1)$ The quadruplet  
$(u(\cdot), \eta_+, \eta_-, W)$ defined on a filtered complete probability space 
$(\Omega, \mathcal{F}, (\mathcal{F}_t)_{t\geq 0}; \P)$ is said to 
be a weak solution of \eqref{sch-1} with its initial value $x$ if all of the following conditions are satisfied: \\
{\rm (i)}
The stochastic process $u(\cdot)\in C((0, T] \times [0,1]; [-1,1])  \cap C([0,1]; {\bf H})\  a.s.$ with 
$u(0)=x$, 
and $f(u) \in L^1([0,T] \times [0,1]) \ a.s.$ for any $T>0$. \\
{\rm (ii)}  $\eta_+$ and $\eta_-$ are two positive random measures on 
$[0, \infty) \times [0,1]$ satisfying  the following property:
$$\eta_{\pm}([\delta, T] \times [0,1])< \infty\ a.s. \
\text{ for all} \ 
\delta \in (0, T]\ \text{ and}\   T>0.$$ 
{\rm (iii)} $(W(t))_{t\geq 0}$ is a cylindrical Wiener process on $L^2(0,1)$. Moreover, the initial value $x$ is independent of
 $(W(t))_{t\geq 0}$  and the stochastic process  $(u(t), W(t))_{t\geq 0}$ is 
$(\mathcal{F}_t)$-adapted.\\
{\rm (iv)} For all $\phi\in D(A^2)$ and  $0 < \delta < t$,
\begin{align}\label{eq-2.1-0507}
\langle u(t), \phi\rangle =& \langle u(\delta), \phi \rangle 
-\frac12 \int_\delta^t \langle u(s), A^2\phi \rangle ds
-\frac12 \int_\delta^t \langle f(u(s)), A \phi \rangle ds    \\
& -\frac12 \int_\delta^t \int_0^1 A\phi(\theta) \eta_+(dsd\theta) +\frac12 \int_\delta^t \int_0^1 A \phi (\theta) \eta_-(dsd\theta)  \notag\\
& +\int_\delta^t \langle B^*\phi, dW(s) \rangle \ \ a.s. \notag
\end{align}
{\rm (v)} The contact properties ${\rm supp}(\eta_+)\subset 
\{(t, \theta) \in [0, \infty) \times [0,1]: \ u(t, \theta)=+1\}$ and 
${\rm supp}(\eta_-)\subset 
\{(t, \theta) \in [0, \infty) \times [0,1]: \ u(t,\theta)=-1\}$ hold almost surely, that is,
$$ \int_0^\infty \int_0^1 (1-u(t, \theta))\eta_+ (dtd \theta)= \int_0^\infty \int_0^1 (1+u(t, \theta))\eta_- (dtd \theta)=0  \ a.s.$$
$(2)$ A weak solution  $(u(\cdot), \eta_+, \eta_-, W)$ is said to be a strong one if 
the stochastic process $(u(t))_{t\geq 0}$ is adapted to the natural filtration $(\mathcal{F}_t)_{t\geq 0}$ generated by $(W(t))_{t\geq 0}$.  
\end{defn}
The term  $ \langle f(u(s)), Ah \rangle$ appearing in
 the right hand side of \eqref{eq-2.1-0507} should be understood in a duality between $L^1$ and $L^\infty$. In fact, it is assumed that $f(u(t)) \in L^1([0,T]\times [0,1]) \ a.s.$ for any fixed $T$ in 
{\rm(i)}. 
 In addition, for the uniqueness of the solution, we mean the pathwise 
uniqueness, that is,
for any two solutions $(u^i, \eta_+^i, \eta_-^i, W), i=1,2$ of 
\eqref{sch-1} with same initial data defined on the same probability 
space with same $W$, then $(u^1, \eta_+^1, \eta_-^1) =(u^2, 
\eta_+^2, \eta_-^2)\ a.s.$

Now let us summarize main results obtained in \cite{DeGo-11}, which will be used in the following.  For brevity, in this section,  we will use the same notations  introduced in Section \ref{sec-3}. To
emphasize the initial value, $u(t; x)$ or $u(t, \cdot;x)$ will be used
according to purposes in the sequel.
\begin{thm} \label{thm-2.1-0506}
For any  $c\in (-1,1)$ and $x\in K:=\{x\in L^2(0,1): x\in [-1,1]\}$ with 
$\bar{x}=c$, the SPDE \eqref{sch-1} has a unique strong solution $(u(\cdot; x); \eta_+, \eta_-, W)$. Moreover,  the following hold:\\
{\rm (i)} The mass of $u(t; x)$ is conservative in $t$, that is, $\bar{u}(t; x) =\bar{x}$ for all $t>0$.\\
{\rm(ii)} $(u(t;x); t\geq0, x\in K\cap {\bf H}^c)$ is a $K\cap {\bf H}^c$-valued continuous Markov 
process and its associated Markov transition semigroup $P_t^c$ is strong Feller on ${\bf H}^c$. 
\\
{\rm(iii)} For each $c\in(-1,1)$, 
$$\nu^c(dx) =\frac{1}{Z^c}\exp 
\left(-\int_{0}^1 F(x(\theta))d \theta \right)1_{K}(x) \mu_c(dx)$$ 
is the unique invariant measure of $P_t^c$,
where $\mu^c$ denotes the Gaussian measure $N(ce_0, (-A)^{-1})$ 
and $Z^c$ denotes the normalization constant.
\\
{\rm(iv)}  For any $k\in  \mathbb{N}$ and $0=t_0 < t_1 < t_2 < \cdots < t_k$, $(u^n(t_i;x))_{i=1}^k$
converges weakly to $(u(t_i;x))_{i=1}^k$ as $n \to \infty$. 
In particular, for any $\phi \in B_b({\bf H}^c)$ and $t\geq 0$, 
we have 
$\lim_{n\to \infty} P_t^{n, c}\phi(x)=P_t^c\phi(x)$. 
Hereafter, $u^n(t;x)$ and
$P_t^{n, c}$ denote the solution of  \eqref{spde-1} with $B=\frac{d}{ d\theta}$ and its associated Markov semigroup.  
\end{thm}

\begin{rem}\label{rem-3.1-190830}
{\rm (i)}   $(-A)^{-1}$ appearing in {\rm (iii)} denotes the inverse of 
$-A$ from $L_0^2$ to $L_0^2$. From Lemma 2.1 \cite{DeZa-07},
it is known that $\mu^c$ is the distribution of the Gaussian process $(B(\theta) -\bar{B}+c)_{\theta\in [0,1]}$
on $C([0,1])$, where $(B(\theta))_{\theta\in [0,1]}$ denotes a standard Brownian motion and $\bar{B} =\int_0^1 B(\theta)d\theta$.
\\
{\rm (ii)} Noting that $\frac{d}{d\theta}\dot{W}(t,\theta)$ and 
$(-A)^\frac12\dot{W}(t,\theta)$ have the 
same covariance structure, we see that it is equivalent for us to
consider $B=(-A)^\frac12$ in \eqref{sch-1}  instead of 
$\frac{d}{d\theta}$ and note that $(-A)^\frac12$ is symmetric. 
So, for simplicity, we will consider $B=(-A)^\frac12$ 
in the sequel and we know that Theorem 
\ref{thm-2.1-0506} still holds.
\end{rem}

\begin{lem}\label{lem-2.2-0507}
Let $B=(-A)^\frac12$. Then $B$ is reversible on 
${\rm span}\{e_i: i=1,2,\cdots\}$ and 
\begin{align*}
|B^{-1}z|^2 =|z|_{-1}^2, \ z\in {\bf H}^0.
\end{align*}
\end{lem}
\begin{proof}
Recalling the definition of the operator $A$ and the seminorm $|\cdot |_{\gamma}$, we can easily proof this lemma.
\end{proof}

The following is the main result of this section. Since the mass 
of the solution to \eqref{sch-1} is required to be conserved, the 
well-known approaches used for the stochastic partial 
differential equation with additive noise, see \cite{Wan-07,
 Wan-17, Xie-19,  Zh-10} for example,  
can not applied to our case.  Moreover, the case of double-well potential is covered. 
To show our main result, we make use of  the approach 
initially introduced in \cite{Wan-11}, in which the stochastic different equations with 
multiplicative noise is studied.

\begin{thm} \label{thm-2.1} 
Suppose $\pi^2> \lambda$. Then the Harnack inequality with power $p>1$ 
\begin{align} \label{eq-2.1-0504}
|P_t^{c}\phi |^p(y) \leq   P_t^{c}|\phi|^p (x) 
\exp\left\{\frac{p (\pi^2-\lambda)\pi^2 |x-y|_{-1}^2 }{2(p-1)(e^{(\pi^2-\lambda)\pi^2 t}-1) } \right\} 
\end{align}
holds for any $\phi \in B_b({\bf H}^c)$, $x, y\in K\cap {\bf H}^c$
and $t>0$. 
In particular,
the log-Harnack inequality 
\begin{align} \label{eq-2.2-0504}
P_t^{c}\log \phi(y) \leq   \frac{ (\pi^2-\lambda)\pi^2 |x-y|_{-1}^2 }{2(e^{(\pi^2-\lambda)\pi^2 t}-1) } + \log P_t^{c}\phi(x)
\end{align}
holds for any $0<\phi \in B_b({\bf H}^c)$, $x, y\in K\cap {\bf H}^c$
and $t>0$.

\end{thm}

\begin{proof}
Let us fix $T>0$ and let  $\gamma (t)$ be a continuously differentiable and strictly positive function on 
$[0, T)$ with $\gamma(T)=0$, which be specified later.  Let  $\aleph$ denote the projection of $\bf{H}$ to 
${\rm  span}\{e_i: i=1,2,3, \cdots\}$ and then consider the coupling stochastic partial differential equation
\begin{align}\label{eq-2.1-0501} 
\left\{
\begin{aligned}
dw^n(t)= 
& -\frac{1}{2} A
\big\{ A w^n (t) -p_n(w^n(t)) +\lambda  w^n(t) \big\} dt + \frac{\aleph(u^n(t) -w^n(t))}{\gamma(t)}dt \\
& + Bd{W}(t), \ \  t\in [0, T),\\
w^n(0) = &y,
\end{aligned}
\right.
\end{align}
where $(u^n(t))_{t\geq 0}$ denotes the solution of \eqref{spde-2}
with $B=(-A)^\frac12$.

Since $\aleph$ is a bounded linear operator, by following the 
arguments used in  \cite{DaDe-96}, one can show that 
for each initial value $y\in {\bf H}$, the SPDE \eqref{eq-2.1-0501}  has a unique solution $w^n$ up to the explosion time $\sigma^n$ such that 
$w^n \in C([0, \sigma^n \wedge T); H) \cap L^{2n+2}((0, \sigma^n \wedge T) \times (0,1))$ a.s., where 
$\sigma^n:= \lim_{k\to \infty} \sigma_k^n$ with 
$\sigma_k^n=\inf\{t\in[0,T): |w^n(t)|_{-1} \geq k\}$.

Moreover, the conservation of the average of $w^n(t)$ holds for 
$t\in [0, \sigma^n \wedge T)$. Indeed,  considering the mild  solution  of \eqref{eq-2.1-0501}, 
we have that  for any 
$x\in L^2(0,1)$ with $\bar{x}=c\in (-1,1)$ and $t\leq  \sigma_k^n \wedge T$, 
\begin{align*}
& \langle w^n(t), e_0 \rangle  \\
= &\langle e^{-\frac12 A^2t} x, e_0\rangle  +
\int_0^t \left\langle A e^{-\frac12 A^2(t-s)}
 \big[p_n(w^n(s)) -\lambda w^n(s) \big], e_0 \right\rangle ds  \\
& +
\int_0^t \left\langle e^{-\frac12 A^2(t-s)}  
\frac{ \aleph(u^n(s) -w^n(s))}{\gamma(s)},  e_0 \right\rangle ds 
+ \int_0^t \left\langle B e^{-\frac12 A^2(t-s)}e_0, dW(s) \right\rangle.  \notag
\end{align*}
Now noting that $ e^{-\frac12 A^2t}e_0 =e_0$ and $B e_0= \aleph e_0=0$, we obtain that  
$$\langle w^n(t), e_0 \rangle=\langle  x, e_0\rangle=c,\  t \in [0, \sigma_k^n \wedge T),$$
which clearly implies our claim by the density of $L^2$ in ${\bf H}$.
From now on, the proof will divided into three steps.

\noindent {\it Step 1:}  The goal of this step is to construct  a successful 
coupling up to time $T$. More precisely,  we will  show that $w^n(T;y) =u^n(T;x)$ holds almost surely under a probability measure  equivalent to $\P$. 

To show it, let us set $Y^n(t) =u^n(t) -w^n(t), \ t\leq  \sigma_k^n \wedge T$ and let $R\in (0,T)$ be fixed. Then by the conservation of the mass, we have that $\overline{Y^n}(t) =0$ whenever $x, y\in {\bf H}^c$ and 
$Y^n(t)$ satisfies 
\begin{align}\label{eq-2.2-0501} 
\left\{
\begin{aligned}
dY^n(t)= 
& -\frac{1}{2} A
\big\{A Y^n(t) -[p_n(u^n(t)) - p_n(w^n(t)) ] +\lambda  Y^n(t) \big\}dt \\
& - \frac{\aleph Y^n(t)}{\gamma(t)}dt,  
\ \  t\in [0,  \sigma_k^n \wedge r), \\
Y^n(0) = &x- y.
\end{aligned}
\right.
\end{align}
Then, using the increasing property of  $p_n$ and $\overline{Y^n}(t) =0$, we can deduce  analogously to \eqref{eq-1.18-0501} that 
\begin{align}\label{eq-2.3-0501}
d|Y^n(t)|_{-1}^2 \leq  & -|Y^n(t)|_{1}^2dt  + \lambda |Y^n(t)|^2dt
- \frac{2\langle (-A)^{-1} \aleph Y^n(t), Y^n(t) \rangle}{\gamma(t)}dt\\
=&  -|Y^n(t)|_{1}^2dt  + \lambda |Y^n(t)|^2dt
- \frac{2|\aleph Y^n(t)|_{-1}^2}{\gamma(t)}dt   \notag \\
\leq & -(\pi^2-\lambda)  |Y^n(t)|^2dt  -
 \frac{2|\aleph Y^n(t)|_{-1}^2}{\gamma(t)}dt \notag \\
\leq &  -(\pi^2-\lambda)\pi^2  |Y^n(t)|_{-1}^2dt 
 - \frac{2| Y^n(t)|_{-1}^2}{\gamma(t)}dt, \notag \  t\in [0,  \sigma_k^n \wedge R),
\end{align} 
where the assumption $\pi^2 >\lambda$ has been used for last inequality.

Hence, \eqref{eq-2.3-0501} and the chain rule give that
\begin{align}\label{eq-2.4-0501}
d \frac{|Y^n(t)|_{-1}^2}{ \gamma(t)} \leq 
-\frac{|Y^n(t)|_{-1}^2}{\gamma^2 (t)} 
\big(\gamma'(t) +(\pi^2-\lambda)\pi^2 \gamma(t) +2  \big)dt, \  t\in [0,  \sigma_k^n \wedge R],
\end{align}
where the strict positivity of $\gamma(t)$ has been used.

Now let us specify the function  $\gamma(t)$.
Let $\alpha \in (0, 2)$ and $\gamma (t)$ be the unique solution 
of the ordinary differential equation
\begin{align*}
\gamma'(t) +(\pi^2-\lambda)\pi^2 \gamma(t) +2 = \alpha
\end{align*}
with $\gamma(T)=0$, that is,
\begin{align} \label{eq-2.9-0507}
\gamma(t)=\frac{2-\alpha}{(\pi^2-\lambda)\pi^2}\left(e^{(\pi^2-\lambda)\pi^2(T-t)}-1 \right),  \ t\in [0,T].
\end{align}
It is easy to testify that $\gamma(t),  t\in [0,T]$ has all of the properties stated at the beginning of the proof.

By noting that $\alpha \in (0, 2)$ and using \eqref{eq-2.4-0501}, we easily  see that 
\begin{align}\label{eq-2.7-0502}
\int_0^t\frac{|Y^n(s)|_{-1}^2}{\gamma^2(s)}ds +\frac{|Y^n(t)|_{-1}^2}{\alpha\gamma(t)} \leq \frac{|x-y|_{-1}^2}{\alpha \gamma(0)}, \  t\in [0,  \sigma_k^n \wedge R].
\end{align}
Let us define the stochastic process $N(t),  t\in [0,  \sigma^n \wedge R]$ by
\begin{align}\label{eq-2.8-0502}
N(t) =&\exp \left(-\int_0^t  \left\langle \frac{B^{-1} \aleph(u^n(s) -w^n(s))}{\gamma(s)}, dW(s) \right\rangle \right. \\
& \qquad \qquad  \qquad -\left. \int_0^t \frac{|B^{-1} \aleph(u^n(t) -w^n(t))|^2}{2\gamma^2(s)} ds\notag
\right). 
\end{align}
Thanks to \eqref{eq-2.7-0502} and Lemma \ref{lem-2.2-0507}, we have  that for all $x, y\in {\bf H}^c$ and $t \in [0,  \sigma_k^n \wedge R]$
\begin{align}\label{eq-2.9-0502}
\int_0^t \frac{|B^{-1} \aleph(u^n(s) -w^n(s))|^2}{2\gamma^2(s)}ds
\leq \frac{|x-y|_{-1}^2}{2\alpha \gamma(0)}.
\end{align}

Let us now define the stochastic process $\overline{W}(t), \ t\in [0,  \sigma^n \wedge R]$ by
\begin{align} \label{eq-2.10-0502}
d\overline{W}(t) =dW(t)+ \frac{B^{-1} \aleph(u^n(t) -w^n(t))}{\gamma(t)}dt.
\end{align}
Then by the Novikov condition and the Girsanov theorem, 
we know that  $\overline{W}(t), t\in [0,  \sigma_k^n \wedge R]$ is  a cylindrical Wiener process on $L^2(0,1)$ under the probability measure $N(\sigma_k^n \wedge T)\P$.

By the definitions of $N(t)$ and $\overline{W}(t)$ and by noting \eqref{eq-2.9-0502}, we have
\begin{align*}
& \log N(t) \\ 
=& 
-\int_0^t  \left\langle \frac{B^{-1} \aleph(u^n(s) -w^n(s))}{\gamma(t)}, d \overline{W}(s) \right\rangle  
+ \int_0^t \frac{|B^{-1} \aleph(u^n(t) -w^n(t))|^2}{2\gamma^2(s)} ds 
\notag \\
\leq &  -\int_0^t  \left\langle \frac{B^{-1} \aleph(u^n(s) -w^n(s))}{\gamma(t)}, d \overline{W}(s) \right\rangle   
+ \frac{|x-y|_{-1}^2}{2\alpha \gamma(0)}, \ t \in [0,  \sigma_k^n \wedge R]. \notag
\end{align*}
Therefore, by taking the expectations of both sides of the above inequality with respect to $N(\sigma_k^n \wedge T)\P$, we obtain
\begin{align}\label{eq-3.13-0831}
\mathbb{E}\left[N(\sigma_k^n \wedge R)\log N(\sigma_k^n \wedge R) \right]  \leq \frac{|x-y|_{-1}^2}{2\alpha \gamma(0)}. 
\end{align}
Recalling that $R\in[0,T)$ is arbitrary, we have that $N(\sigma_k^n \wedge R), R\in[0,T)$ is uniformly integrable and 
\begin{align} \label{eq-2.11-0502}
\sup_{R\in [0,T)}\sup_{k, n\in \mathbb{N}}\mathbb{E} \left[N(\sigma_k^n \wedge R)\log N(\sigma_k^n \wedge R) \right] \leq  \frac{|x-y|_{-1}^2}{2\alpha \gamma(0)}.
\end{align}

Then by the martingale convergence theorem and the Doob optional sampling theorem, it follows that $N(t\wedge \sigma^n), t\in [0,T]$ is a martingale and by letting $k\to \infty$ in \eqref{eq-3.13-0831}, 
\begin{align} \label{eq-2.11-0502-1}
\sup_{n \in \mathbb{N}}\mathbb{E} \left[N(\sigma^n \wedge t)
\log N(\sigma^n \wedge t)  \right] 
\leq  \frac{|x-y|_{-1}^2}{2\alpha \gamma(0)}, \ t\in [0,T].
\end{align}
In addition, we known that  $(\overline{W}(t))$ is  a cylindrical Wiener process on $L^2(0,1)$ under the probability measure $\mathbb{Q}:=N(\sigma^n \wedge T)\P$
 up to time $\sigma^n \wedge R$.  By \eqref{eq-2.7-0502}, in fact we can show that 
for all $n \in \mathbb{N}$, $\sigma^n= T\ \mathbb{Q}$-$a.s.$ Indeed, since $(u^n(t))_{t \geq 0}$ is the global solution of \eqref{spde-2}, we see that
$\tau_l^n=\inf\{t\geq 0: |u^n(t)|_{-1}\geq l\}$ diverges to $\infty$ as 
$l\to \infty$. Noting that $\gamma(t)$ is decreasing with respect to  $t\in [0,T]$ and $$|Y^n(t\wedge \tau_{k}^n \wedge \sigma_{2k}^n)|_{-1}\geq k,$$ we have
\begin{align}\label{eq-3.16-190831}
\mathbb{E}^\mathbb{Q} \left[1_{\{\sigma_{2k}^n\leq t <\tau_k^n \}} \frac{|Y^n(t\wedge \tau_{k}^n \wedge \sigma_{2k}^n)|_{-1}^2}{ \gamma(t\wedge \tau_k^n \wedge \sigma_{2k}^n)}\right] 
\geq \frac{k^2}{\gamma(0)}\mathbb{Q}(\sigma_{2k}^n\leq t <\tau_k^n ).
\end{align}
On the other hand, by \eqref{eq-2.7-0502}, it is known that 
the left hand of \eqref{eq-3.16-190831} is bounded from above by 
$ \frac{|x-y|_{-1}^2}{\gamma(0)}.
$
Hence, letting  now  $k\to \infty$  in 
\eqref{eq-3.16-190831}, we  obtain
$$ \mathbb{Q}(\sigma^n \leq t) =0,\  t\in[0,T),$$ which clearly  implies  $\mathbb{Q}(\sigma^n = T)=1$.

Consequently, in the sequel, we can write $d\mathbb{Q}=N(T)d\P$ and
then we know that
$(\overline{W}(t))_{ t\in [0,T]}$ defined by \eqref{eq-2.10-0502} is 
a cylindrical Wiener process on $L^2(0,1)$ with respect to $\mathbb{Q}$. 

Using  the cylindrical Wiener process  $(\overline{W}(t))_{ t\in [0,T]}$, we easily see that the SPDE \eqref{eq-2.1-0501} can be rewritten as follows:
\begin{align} \label{eq-2.12-0502}
\left\{
\begin{aligned}
dw^n(t)= 
& -\frac{1}{2} A
\big\{ A w^n (t) -p_n(w^n(t)) +\lambda  w^n(t) \big\}dt  
+ Bd{\overline{W}}(t), \ \  t\in [0, T),\\
w^n(0) = &y \in {\bf H}^c,
\end{aligned}
\right.
\end{align}

Since under $\mathbb{Q}$,  $\overline{W}(t), t\in [0,T]$ is a cylindrical
 Wiener process on $L^2(0,1)$,  similarly to \eqref{spde-2}, we know
that \eqref{eq-2.12-0502} has global unique solution 
$w^n\in C([0, T]; H) \cap L^{2n+2}((0, T) \times (0,1))$. 
Moreover, the distribution of $w^n(t)$ under $\mathbb{Q}$ is same as that of 
$u^n(t; x)$ under $\P$ by the uniqueness in law of solutions.  
Therefore, by the equivalence of $\mathbb{Q}$ and $\P$, we know that \eqref{eq-2.1-0501}  also has the global solution up to time $T$.

From now on, we claim that   the coupling of  $\eqref{spde-2}$ and $\eqref{eq-2.1-0501}$ is made successfully up to time $T$. 
Let $\tau$ denote the coupling time, that is,
$$
\tau=\inf\{t\in [0,T]: u^n(t)=w^n(t)\ {\rm in}\  {\bf H}^c \}.
$$
with the convention $\inf \emptyset =\infty.$
Then we can show $\tau \leq T\ a.s.$ by contradiction. In fact, if 
$\tau(\omega) >T$, then 
$$
\inf_{t\in [0,T]} |u^n(t, \omega) - w^n(t, \omega)|_{-1}^2
$$
is  strictly positive, since both $u^n$ and $w^n$ are continuous 
stochastic processes with values in ${\bf H}^c$. Hence,
 we obtain that the integral of $\frac{|u^n(t,\omega) -w^n(t, \omega)|_{-1}^2}{\gamma(t)}$ on [0,T] diverges,  by noting that 
$\int_0^T\frac{1}{\gamma(t)}dt=\infty$. Therefore,  we have that on the set $\{\tau >T\}$, 
\begin{align}\label{eq-2.4-0502}
\int_0^T\frac{|Y^n(s)|_{-1}^2}{\gamma(s)}ds =\infty.
\end{align}
On the other hand, 
noting that $(\pi^2-\lambda)\pi^2 >0$, we obtain by \eqref{eq-2.3-0501}  that 
\begin{align*}
\int_0^t \frac{| Y^n(s)|_{-1}^2}{\gamma(s)}ds \leq \frac{|x-y|_{-1}^2}{2}, \  t\in [0,  \sigma_k^n \wedge R],
\end{align*}
which  contradicts with \eqref{eq-2.4-0502}.  Consequently, our
claim is proved.  In particular, we have $$w^n(T;y) =u^n(T;x)\ \ \mathbb{Q}\text{-}a.s.$$

Based on the above preparations, this theorem can be shown in 
the usual way \cite{Wan-13}. For the reader's convenience, we give the outline of the proof.  

\noindent {\it Step 2:} Let us formulate the proof of \eqref{eq-2.1-0504}. We first show for any $q>1$, 
\begin{align} \label{eq-3.19-190831}
\mathbb{E} [|N(t)|^q] \leq \exp \left\{\frac{(q-1)q|x-y|_{-1}^2}{ 2\alpha \gamma(0)} \right\},\ t\leq T.
\end{align}
 By  the definitions of $N(t)$ and $W(t)$, it follows easily that for any $q>1$, 
\begin{align}\label{eq-2.17-0505}
\mathbb{E} [|N(t)|^q] = &\mathbb{E}^{\mathbb{Q}}[|N(t)|^{q-1}]\\
                 =& \mathbb{E}^{\mathbb{Q}} 
\left[\exp \left\{-(q-1)\int_0^t  \left\langle \frac{B^{-1} \aleph(u^n(s) -w^n(s))}{\gamma(s)}, d{W}(s) \right\rangle \right. \right.    \notag  \\
& \qquad \qquad  \qquad -\left. \left.  (q-1)\int_0^t \frac{|B^{-1} \aleph(u^n(t) -w^n(t))|^2}{2\gamma^2(s)} ds      \right\}       \right] \notag  \\
 =& \mathbb{E}^{\mathbb{Q}} 
\left[\exp \left\{-(q-1)\int_0^t  \left\langle \frac{B^{-1} \aleph(u^n(s) -w^n(s))}{\gamma(s)}, d\overline{W}(s) \right\rangle \right. \right.    \notag  \\
& \qquad \qquad  \qquad +\left. \left.  (q-1)\int_0^t \frac{|B^{-1} \aleph(u^n(t) -w^n(t))|^2}{2\gamma^2(s)} ds      \right\}       \right], \ t\leq T.    \notag  
\end{align} 
By \eqref{eq-2.9-0502}, we have that 
\begin{align}\label{eq-2.18-0505}
\sup_{t\in[0,T]}\exp\left\{ \int_0^t \frac{|B^{-1} \aleph(u^n(s) -w^n(s))|^2}{2\gamma^2(s)} ds\right\} \leq
\exp \left\{\frac{(q-1)q|x-y|_{-1}^2}{2 \alpha \gamma(0)} \right\}. 
\end{align}
Note that 
\begin{align*}
U(t):=& \exp \left\{-(q-1)\int_0^t  \left\langle \frac{B^{-1} \aleph(u^n(s) -w^n(s))}{\gamma(s)}, d \overline{W}(s) \right\rangle \right. \\
& \qquad \qquad  \qquad \left.  -(q-1)^2\int_0^t \frac{|B^{-1} \aleph(u^n(s) -w^n(s))|^2}{2\gamma^2(s)} ds  \right\}, \ t\leq T \notag
\end{align*}
is an exponential martingale under ${\mathbb{Q}}$. Then, by 
\eqref{eq-2.18-0505}, we have 
\begin{align*}
\mathbb{E} [|N(t)|^q] = &\mathbb{E}^{\mathbb{Q}}
\left[U(t) \exp\left\{ \int_0^t \frac{(q-1)q|B^{-1} \aleph(u^n(s) -w^n(s))|^2}{2\gamma^2(s)} ds\right\} \right] \\
\leq &\mathbb{E}^{\mathbb{Q}}
\left[U(t) \sup_{t\in [0.T]}\exp\left\{ \int_0^t \frac{(q-1)q|B^{-1} \aleph(u^n(s) -w^n(s))|^2}{2\gamma^2(s)} ds\right\} \right] \notag \\
\leq & \exp \left\{\frac{(q-1)q|x-y|_{-1}^2}{2 \alpha \gamma(0)} \right\} \mathbb{E}^{\mathbb{Q}}[U(t)] \notag \\
= &\exp \left\{\frac{(q-1)q|x-y|_{-1}^2}{2 \alpha \gamma(0)} \right\},  \ t\leq T, \notag
\end{align*}
where
\eqref{eq-2.9-0502} has been used for the second inequality.
Therefore, the proof of \eqref{eq-3.19-190831} is completed.

Let us now formulate the proof \eqref{eq-2.1-0504}. According to the relation between $w^n(t;y)$ and $u^n(t;x)$, we have that for any $p>1$,  any $\phi \in B_b({\bf H})$ and any $x, y\in K\cap {\bf H}^c$
\begin{align*}
|P_T^{n,c}\phi|^p(y) = &  |\mathbb{E}^\mathbb{Q}[\phi(w^n(T;y))]|^p \\
=  & |\mathbb{E}^\mathbb{Q}[\phi(u^n(T;x))]|^p \\
=& |\mathbb{E}[N(T)\phi(u^n(T;x))]|^p \\
\leq &  \mathbb{E}[N(T)^{\frac{p}{p-1}}]^{p-1} P_T^{n,c}|\phi|^p(x)\\
\leq & P_T^{n,c}|\phi|^p(x)  \exp \left\{\frac{p|x-y|_{-1}^2}{2\alpha(p-1) \gamma(0)} \right\},
\end{align*}
where \eqref{eq-3.19-190831} with $q= \frac{p-1}{p}$ has been used for the last inequality.

Consequently,  we can  complete the proof of \eqref{eq-2.1-0504} by letting $\alpha =1$ and then $n \to \infty$ thanks to  Theorem \ref{thm-2.1-0506}.

\noindent {\it Step 3:} Let us finally give the proof \eqref{eq-2.2-0504} in brief. By the definition of $\mathbb{Q}$, the Young inequality \eqref{eq-1.27-0504} and the estimate \eqref{eq-2.11-0502-1}, it follows that 
\begin{align}\label{eq-2.15-0504}
P_T^{n,c}\log \phi(y)=&\mathbb{E}^{\mathbb{Q}} [\log\phi(w^n(T;y)) ] \\
                    =&\mathbb{E} [N(T)\phi(u^n(T; x)) ]  \notag\\
                    \leq & \mathbb{E} [N(T) \log N(T)] +\log \mathbb{E}[\phi(u^n(T;x))] \notag \\
\leq  & \frac{|x-y|_{-1}^2}{2 \alpha \gamma(0)} +\log P_T^{n,c} \phi(x). \notag
\end{align}
Recalling the representation of $\gamma$, see 
\eqref{eq-2.9-0507},  and minimizing the first term in  \eqref{eq-2.15-0504} with respect to $\alpha \in (0,2)$,  we see that
\begin{align*}
P_T^{n,c} \log \phi(y) \leq   \frac{ (\pi^2-\lambda)\pi^2 |x-y|_{-1}^2 }{2(e^{(\pi^2-\lambda)\pi^2 T} -1)} + \log P_T^{n,c}\phi(x).
\end{align*}
Now thanks to Theorem \ref{thm-2.1-0506}, we can easily  complete the proof of \eqref{eq-2.2-0504} with $t=T$
by letting $n \to \infty$ in the above inequality.  
\end{proof}

\vskip 0.3cm
\begin{rem}
If we consider $B=\frac{d}{d\theta}$ with ${\rm Dom}(B)= {H^1(0,1)}$ as that in the original paper \cite{DeGo-11}, then we can show the following equation  
\begin{align*}
|B^* (BB^*)^{-1} z|^2 =|z|_{-1}^2, \ z\in {\bf H}^0,
\end{align*}
by noting that $BB^* =-A$.  Thus, we can replace $B$ in the definition of $N(t)$, see \eqref{eq-2.8-0502}, by $B^*(BB^*)^{-1}$ and then obtain the 
same results  as those in Theorem  \ref{thm-2.1}.

In addition, the method used  in Theorem \ref{thm-2.1}  can be also
applied to the SPDE \eqref{sch-1}  with more general $B$ instead of 
$B=\frac{d}{d\theta}$ or $B=(-A)^\frac12$. In fact, if $BB^*$ is 
reversible restricted on $${\rm span}\{e_n: n=1,2,\cdots \}$$ and 
$$|B^*BB^* z|\leq C|z|_{-1}, z\in {\bf H}$$ for some $C>0$ and  {\rm (i), (ii), (iv)} in Theorem 
\ref{thm-2.1-0506} hold, then the Harnack equalities similar as those
in Theorem  \ref{thm-2.1} can be established. 
For example, if there exists a strictly positive sequence $\{b_n\}_{n=1}^\infty$ such that $Be_n=b_n e_n, n=1,2,\cdots$ and the sequence 
$\{n b_n^{-1}\}_{ n=1}^\infty$ is bounded, then $B$ satisfies the assumptions stated above.

\end{rem}

According to Theorem
1.4.1 \cite{Wan-13},  many important properties of $P_t^c$ can be 
deduced from Theorem \ref{thm-2.1}. For example, uniqueness of invariant probability measures can be easily known. As we stated in  Theorem \ref{thm-2.1-0506}, the existence and uniqueness of invariant  measures has been proved in  \cite{DeGo-11} by 
a different approach. 
Here, it is valuable to know that it can be reproved as the application of Harnack inequalities. Moreover, we
also know that $P_t^c$ is absolutely continuous with respect to its invariant 
measure $\nu^c$ and the following results hold for the density 
$p^c(t,x,y)$ of $P_t^c$ with respect to $\nu^c$.
\begin{cor} 
Under the assumptions of Theorem \ref{thm-2.1}, the following heat
 kernel inequalities are fulfilled  for any $t>0, x, y\in {\bf H}^c$ and $p>1$ 
\begin{align*}
 & \int_{ {\bf{H}}^c}p^c(t,x,z)\left\{\frac{p^c(t,x,z)}{p^c(t,y,z)} \right\}^{\frac1{p-1}}\nu^c(dz) \leq
\exp\left\{\frac{p (\pi^2-\lambda)\pi^2 |x-y|_{-1}^2 }{2(p-1)^2(e^{(\pi^2-\lambda)\pi^2 t} -1)} \right\},  
\\
 & \int_{{\bf{H}}^c}p^c(t,x,z)\log \frac{p^c(t,x,z)}{p^c(t,y,z)}\nu^c(dz)
 \leq 
 \frac{ (\pi^2-\lambda)\pi^2 |x-y|_{-1}^2 }{2(e^{(\pi^2-\lambda)\pi^2 t}
-1) } .
\end{align*}
\end{cor}

\vskip 0.5cm
\begin{ackn}

L. Gouden\`ege was supported in part by the French National Research Agency (ANR) as leader of the SIMALIN project ANR-19-JCJC, and  B. Xie was supported in part by
Grant-in-Aid for Scientific Research (C) 16K05197 from Japan Society for the Promotion of Science (JSPS). 
The authors warmly thank the  referee for  invaluable comments and suggestions, which led to significant improvement in the presentation of the paper
\end{ackn}

\bibliographystyle{plain}

\end{document}